\documentclass{article}
\usepackage{caption}

\usepackage[paper=letterpaper,margin=1in]{geometry}

\usepackage{graphicx}
\usepackage{amsmath,amsthm}

\newcommand{\gp}[1]{\gamma_P\left(#1\right)}
\newcommand{\vftk}[1]{{\gamma}_P^k\left(#1\right)}
\newcommand{\gpkset}{\ddot{\gamma}_P^k\text{-set}}
\newcommand{\gpk}[1]{\ddot{\gamma}_P^k\left(#1\right)}
\newcommand{\pmusset}[2]{\operatorname{\#PMU}_{#2}\left(#1\right)}
\newcommand{\pmus}[1]{\operatorname{\#PMU}\left(#1\right)}
\newcommand{\vftother}[2]{{\gamma}_P^{#1}\left(#2\right)}
\newcommand{\gpkother}[2]{\ddot{\gamma}_P^{#1}\left(#2\right)}
\newcommand{\gpkotherset}[1]{\ddot{\gamma}_P^{#1}\text{-set}}
\newcommand{\gpkplus}[1]{\ddot{\gamma}_P^{k+1}\left(#1\right)}
\newcommand{\bigpds}[1]{\mathsf{s}\left(#1\right)}
\newcommand{\bigpdsj}[1]{\mathsf{s}_j\left(#1\right)}
\newcommand{\bigpdsjother}[2]{\mathsf{s}_{#2}\left( #1\right)}

\newtheorem{thm}{Theorem}
\newtheorem{cor}{Corollary}
\newtheorem{lem}{Lemma}
\newtheorem{prop}{Proposition}
\newtheorem{obs}{Observation}
\newtheorem{quest}{Question}
\newtheorem{defn}{Definition}
\newtheorem{ex}{Example}

\title{An Introduction to PMU-Defect-Robust Power Domination: Bounds, Bipartites, and Block Graphs}
\author{Beth Bjorkman \thanks{Air Force Research Laboratory, Wright-Patterson Air Force Base, OH} \and Esther Conrad \thanks{Iowa State University, Ames, IA.} \and Mary Flagg \thanks{University of St. Thomas, Houston, TX(flaggm@stthom.edu)}}
\date{December 12, 2023}

\begin{document}

\maketitle

\abstract{Sensors called phasor measurement units (PMUs) are used to monitor the electric power network. The power domination problem seeks to minimize the number of PMUs needed to monitor the network. We extend the power domination problem and consider the minimum number of sensors and appropriate placement to ensure monitoring when $k$ sensors are allowed to fail with multiple sensors allowed to be placed in one location. That is, what is the minimum multiset of the vertices, $S$, such that for every $F\subseteq S$ with $|F|=k$, $S\setminus F$ is a power dominating set. Such a set of PMUs is called a \emph{$k$-PMU-defect-robust power domination set}. This paper generalizes the work done by Pai, Chang and Wang in 2010 on fault-tolerant power domination, which did not allow for multiple sensors to be placed at the same vertex. We provide general bounds and determine the $k$-PMU-defect-robust power domination number of some graph families.}

\textbf{Keywords---} Power Domination, Robust, Fault-Tolerant

\maketitle

\section{Introduction}

The \textit{power domination problem} seeks to minimize the number of sensors called phasor measurement units (PMUs) used to monitor the electric power network, defined by Haynes et al. in  \cite{hhhh02}. The electrical power grid is represented as a graph in which each vertex represents an electrical node and each edge represents a transmission line joining two electrical nodes. The PMUs are placed on the vertices and observation rules are applied to both the vertices and the edges of the graph:  A PMU monitors the vertex where it is placed as well as the adjacent vertices and the edges incident to the vertex. Ohm's law and Kirchoff's law are used to solve a system of equations to monitor one unknown edge and one unknown vertex at a time. This process, as defined in \cite{hhhh02}, consisted of five rules focusing on the observation of both the edges and the vertices. In \cite{bh05}, Brueni and Heath simplified the process to two rules that focus solely on the observation of the vertices. The vertex-only observation process was shown to be equivalent to the original definition.

Pai, Chang, and Wang \cite{pcw10} generalized power domination to create \emph{fault-tolerant power domination} in 2010 to model the possibility of sensor failure. The $k$-\textit{fault-tolerant power domination problem} seeks to find the minimum number of PMUs needed to monitor a power network (and their placements) given that any $k$ of the PMUs will fail. The vertex containing the failed PMU remains in the graph, as do its edges; it is only the PMU that fails. This generalization allows for the placement of only one PMU per vertex.

We consider the related problem of the minimum number of PMUs needed to monitor a power network given that $k$ PMUs will fail \emph{but also allow for multiple PMUs to be placed at a given vertex}. We call this \emph{PMU-defect-robust power domination}, as it is not the vertices that cause a problem with monitoring the network, but the individual PMUs themselves. This models potential synchronization issues, sensor errors, or malicious interference with the sensor outputs.

To demonstrate the difference between fault-tolerant power domination and PMU-defect-robust power domination and how drastic the difference between these two parameters can be, consider the star on $16$ vertices with $k= 1$, shown in Figure \ref{fig:starmulti}. Notice that in fault-tolerant power domination, if one PMU is placed in the center of the star and this PMU fails, then all but one of the leaves must have PMUs in order to still form a power dominating set. However, with PMU-defect-robust power domination, placing two PMUs in the center is sufficient to ensure that even if one PMU fails, the power domination process will still observe all of the vertices. 

    \begin{figure}[htbp]
    \begin{center}
    \includegraphics[scale=1]{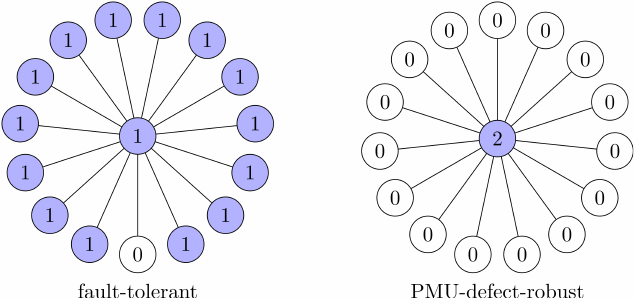}
    \end{center}
    \captionsetup{width=.9\linewidth}
    \caption{A minimum fault-tolerant power dominating set and a minimum PMU-defect-robust power dominating set shown for a star graph when $k=1$}
    \label{fig:starmulti}
    \end{figure}

This paper is organized as follows: In Section \ref{sec:prelim} we give formal definitions of basic graph theory terminology,  and of the power domination and PMU-defect-robust power domination processes. Section \ref{sec:bounds} presents basic bounds for the $k$-PMU-defect-robust power domination number and improve these bounds in the case that minimum (or almost minimum) power dominating sets overlap sufficiently. Complete bipartite graphs are covered in Section \ref{sec:completebipartite} and block graphs in Section \ref{sec:blocks}. We conclude in Section \ref{sec:futurework} with ideas for future work.

\section{Preliminaries}\label{sec:prelim}

A \emph{graph} $G=(V,E)$ is an ordered pair consisting of a finite nonempty set of \emph{vertices}, $V=V(G)$, and \emph{edges}, $E=E(G)$, consisting of  unordered pairs of distinct vertices, usually denoted as $uv$. All graphs will be simple, undirected and finite. If $uv \in E(G)$ then $u$ and $v$ are said to be \emph{adjacent} or \emph{neighbors}. The \emph{neighborhood} of a vertex $v \in V(G)$ is defined to be the set containing all of neighbors of $v$ and is denoted by $N_G(v)=\{w \in V:\{v,w\} \in E(G)\}$. The \emph{closed neighborhood} of $v$ is the set $N[v]=N(v) \cup \{v\}$. The \emph{degree} of a vertex $v$, denoted by $\deg{v}{}$, is the cardinality of $N(v)$. The \emph{minimum degree} of $G$ is defined as $\delta(G)=\min\{\deg{v}{} : v \in V(G)\}$. The \emph{maximum degree} of $G$ is defined as $\Delta(G)=\max\{\deg{v}{} : v \in V(G)\}$.

A \emph{path} between vertices $v_1$ and $v_\ell$ is a sequence of distinct vertices $v_1,v_2,\ldots,v_\ell$ such that $v_iv_{i+1}\in E(G)$ for all $1\leq i \leq \ell -1$. A \emph{cycle} is a sequence of vertices   $v_1,v_2,\ldots,v_\ell,v_1$ such that $v_iv_{i+1}\in E(G)$ for all $1\leq i \leq \ell -1$ and $v_\ell v_i \in E(G)$. A \emph{connected graph} is a graph in which there is a path between any two vertices. A \emph{tree} is a connected graph with no cycles. 

An \emph{induced subgraph} $H$ of a graph $G$ has vertex set $V(H)\subseteq V(G)$ and edge set $E(H)= \{uv : u,v\in V(H), uv\in E(G)\}$. A vertex $v$ in a connected graph $V(G)$ is a \emph{cut vertex} if the induced subgraph $H$ with $V(H)=V(G)\setminus \{v\}$ is not connected.

The \emph{power domination process} proceeds as follows. Given a set $S \subseteq V$ of initial PMU locations, define the set $PDS$ of observed vertices recursively by 
\begin{enumerate}
\item (Domination Step) $PDS = N[S]$
\item (Zero Forcing Step) While there exists $v \in PDS$ such that $v$ has exactly one neighbor $w$ not in $PDS$, $PDS=PDS \cup \{w\}$.
\end{enumerate}

If $PDS=V(G)$ at the end of this process, $S$ is called a \emph{power dominating set}. The \emph{power domination number} of $G$, denoted by $\gp{G}$, is the minimum cardinality of a power dominating set.

A \emph{minimal power dominating set} is a power dominating set $S$ in which the removal of any single vertex from $S$ results in a set that is not a power dominating set. Minimum power dominating sets are minimal, but minimal power dominating sets are not necessarily minimum.

\begin{obs}\label{PDcontainsminimal}
Every power dominating set $S$ contains a minimal power dominating set as a subset.
\end{obs}

In \cite{pcw10}, Pai, Chang, and Wang define the following variant of power domination. For a graph $G$ and an integer $k$ with $0\leq k \leq |V|$, a set $S\subseteq V$ is called a \emph{$k$-fault-tolerant power dominating set of $G$} if $S\setminus F$ is still a power dominating set of $G$ for any subset $F\subseteq V$ with $|F|\leq k$.  The \emph{$k$-fault-tolerant power domination number}, denoted by $\vftk{G}$, is the minimum cardinality of a $k$-fault-tolerant power dominating set of $G$. 

While $k$-fault-tolerant power domination allows us to examine what occurs when a previously chosen PMU location is no longer usable (yet the vertex remains in the graph), PMU-defect-robust power domination allows for multiple PMUs to be placed at the same location and consider if a subset of the PMUs fail. This also avoids issues with poorly connected graphs, such as in Figure \ref{fig:starmulti}. Instead of beginning with a set of vertices to be PMU locations, a \emph{multiset} of PMU locations is used.

\begin{defn}\label{def:gpk}
Let $G$ be a graph, $k \geq 0$ be an integer and let $S$ be a multiset of PMUs with underlying set a subset of $V(G)$. 
The multiset $S$ is a \emph{$k$-PMU-defect-robust power dominating set} ($k$-rPDS)  if for any submultiset $F$ with $|F|=k$, $S \setminus F$ contains a power dominating set of vertices. The minimum cardinality of a $k$-rPDS is denoted $\gpk{G}$. A $k$-rPDS of minimum cardinality is called a $\gpkset$. 
\end{defn} 

Notice that a set $S \subseteq V(G)$ on which PMUs are placed is a multiset in which the multiplicity of $v \in S$ represents the number of PMUs placed on the vertex $v$.  Denote the multiplicity of $v$ in the multiset $S$ of PMUs by $\pmusset{v}{S}$, or when $S$ is clear by  $\pmus{v}$. Given a subset $A \subseteq V$, define 
\[ \pmus A= \sum_{v \in A} \pmus v. \] 

Given a multiset $S$ of vertices of the graph $G$, the ability of $S$ to monitor the graph depends only on whether the underlying set of vertices is a power dominating set, not the multiplicity of each vertex in $S$. If $S$ is a $k$-rPDS, removing any $k$ PMUs must leave an underlying set that contains a minimal (but not necessarily minimum) power dominating set. It will be useful to translate the concept of a minimal power dominating set to a multiset without a minimal condition on multiplicity. 

\begin{defn}\label{essential}
Given a graph $G$, a multiset $M$ of vertices of $G$ is called an \emph{essential multiset} if the underlying set of $M$ is a minimal power dominating set. 
\end{defn}

There are several observations that one can quickly make.

\begin{obs}\label{obs:krPDScontainsessential}
Every $k$-rPDS of a graph $G$ must contain an essential multiset.
\end{obs}

\begin{obs}\label{obs:comparing}
    Let $G$ be a graph and $k\geq 0$. Then
    \begin{enumerate}
        \item $\gpkother{0}{G} = \vftother{0}{G}=\gp{G}$,
        \item $\gpk{G} \leq \vftk{G}$,
        \item $\gp{G}=1$ if and only if $\gpk{G} = k+1$.
    \end{enumerate}     
\end{obs}

\begin{obs}\label{obs:supersetkrobust}
For any graph $G$, if $S$ is a $k$-rPDS of $G$, then any multiset $M$ of vertices of $G$ such that $S \subseteq M$ is also a $k$-rPDS. To see this, note that removing $k$ vertices from $M$ removes at most $k$ vertices from $S$. 
\end{obs}

\begin{obs}\label{removeone}
If $S$ is a $(k+1)$-rPDS, and $v \in V(S)$, then $S\setminus\{v\}$ is a $k$-rPDS since removing $k$ PMUs from $S\setminus\{v\}$ is also removing $k+1$ PMUs from $S$.
\end{obs}

\section{Bounds}\label{sec:bounds}

\subsection{General Bounds}
    
    A useful property of PMU-defect-robust power domination is the subadditivity of the parameter with respect to $k$. This idea is established in the next three statements. 
    
    \begin{prop}\label{prop:incr}
    Let $k\geq 0$. For any graph $G$, $\gpk{G} +1 \leq \gpkplus{G}$.
    \end{prop}
    
    \begin{proof}
    Consider a $\gpkotherset{k+1}$, $S$, of $G$. Let $v\in S$. Create $S'=S\setminus\{v\}$, that is, $S'$ is $S$ with one fewer PMU at $v$. Observe that for any $F'\subseteq S'$ with $|F'|=k$, we have $F'\cup \{v\} \subseteq S$ and $|F'\cup \{v\}|=k+1$. Hence $S\setminus \left( F'\cup \{v\} \right)$ is a power dominating set of $G$. Thus, for any such $F'$, we have $\left( S\setminus \{v\} \right)\setminus F' = S'\setminus F'$ is a power dominating set of $G$. Therefore, $S'$ is a $k$-rPDS of $G$ of size $|S|-1$. 
    \end{proof}
    
    Proposition \ref{prop:incr} can be applied repeatedly to obtain the next result.
    \begin{cor}\label{cor:incrj}
    Let $k\geq 0$ and $j\geq 1$. For any graph $G$, \[\gpk{G} + j \leq \gpkother{k+j}{G}.\]
    \end{cor}
    Corollary \ref{cor:incrj} implies the lower bound in the next proposition. The upper bound follows from taking $k+1$ copies of any minimum power dominating set for $G$ to form a $k$-rPDS.
    \begin{prop}\label{basebounds} Let $k\geq 0$. For any graph $G$,
    \[\gp{G}+k \leq \gpk{G} \leq (k+1)\gp{G}.\]
    \end{prop}

In a similar vein, we can add a power dominating set to a $k$-rPDS in order to create a $(k+1)$-rPDS.
    
    \begin{lem}\label{plusgammap}
For a graph $G$ and a positive integer $k$, $\gpkplus{G} \leq \gpk{G}+\gp{G}.$
\end{lem}

\begin{proof}
For the given graph $G$, let $S$ be a $\gpkset$ and let $D$ be a set of PMUs consisting of one PMU on each vertex of a minimum power dominating set of vertices. Let $C = S \cup D$, and note that $S$ and $D$ need not be disjoint since $C$ is a multiset of PMUs. Removing $k+1$ PMUs from the multiset $C$ either removes all $k+1$ PMUs from $S$, leaving one PMU on each vertex in $D$, or removes at most $k$ PMUs from $S$ and at least one PMU from $D$, which leaves a power dominating set since $S$ is a $\gpkset$. 
\end{proof}
    
We note that for any graph $G$ with $\gp{G} =1$, Observation \ref{obs:comparing}, Proposition \ref{basebounds}, and Lemma \ref{plusgammap} each demonstrate that $\gpk{G} = k+1.$  
    
Haynes et al. observed in \cite[Observation~4]{hhhh02} that in a graph with maximum degree at least 3, a minimum power dominating set can be chosen in which each vertex has degree at least $3$. We observe that this is the same for PMU-defect-robust power domination. 

    \begin{obs}\label{obs:deg3}
    Let $k\geq 0$. If $G$ is a connected graph with $\Delta(G)\geq 3$, then $G$ contains a $\gpkset$ in which every vertex has degree at least 3.
    \end{obs}
    
    A \emph{terminal path} from a vertex $v$ in $G$ is a path from $v$ to a vertex $u$ such that $\deg{u}{}=1$ and every internal vertex on the path has degree 2. A \emph{terminal cycle} from a vertex $v$ in $G$ is a cycle $v,u_1,u_2,\ldots,u_\ell,v$ in which $\deg{u_i}{G}=2$ for $i=1,\ldots, \ell$. 
    
    \begin{prop}\label{prop:twotermpathsortermcycle}
    Let $k\geq 0$ and let $G$ be a connected graph with $\Delta(G)\geq 3$. Let $S$ be a $\gpkset$ in which every vertex has degree at least 3. Any vertex $v\in S$ that has at least two terminal paths from $v$ must have $\pmus{v}=k+1$. Any vertex $v\in S$ that has at least one terminal cycle must have $\pmus{v}=k+1$.
    \end{prop}
    \begin{proof}
        Let $v$ be a vertex in $S$ and suppose that $v$ has two terminal paths or a terminal cycle. All of the vertices in the terminal paths or terminal cycle have degree 1 or 2 and so are not in $S$. Thus, there are at least two neighbors of $v$ which can only be observed via $v$. As $v$ can only observe both of these neighbors via the domination step, it must be the case that $\pmus{v} = k+1$.
    \end{proof}
    
    Zhao, Kang, and Chang \cite{zkc06} defined the family of graphs $\mathcal{T}$ to be those graphs obtained by taking a connected graph $H$ and for each vertex $v\in V(H)$ adding two vertices, $v'$ and $v''$; and two edges $vv'$ and $vv''$, with the edge $v'v''$ optional. The complete bipartite graph $K_{3,3}$ is the graph with vertex set $X\cup Y$ with $|X|=|Y|=3$ and edge set $E=\{xy:x\in X, y\in y\}$.
    
    \begin{thm}{\rm \cite[Theorem~3.]{zkc06}}\label{thm:powdomnover3}
    If $G$ is a connected graph on $n\geq 3$ vertices then $\gp{G}\leq\frac{n}{3}$  with equality if and only if $G\in \mathcal{T}\cup \{K_{3,3}\}$.
    \end{thm}
    
    This gives an upper bound for $\gpk{G}$ in terms of the size of the vertex set and equality conditions, as demonstrated in the next corollary.
    
    \begin{cor}
    Let $G$ be a connected graph with $n\geq 3$ vertices. Then $\gpk{G}\leq (k+1)\frac{n}{3}$ for $k\geq 0$. When $k=0$, this is an equality if and only if $G\in \mathcal{T}\cup\{K_{3,3}\}$. When $k\geq 1$, this is an equality if and only if $G\in\mathcal{T}$.
    \end{cor}
    
    \begin{proof}
    The upper bound is given by Proposition \ref{basebounds} and Theorem \ref{thm:powdomnover3}. From these results, we need only consider $\mathcal{T}\cup \{K_{3,3}\}$ for equality. The $k=0$ case follows directly from the power domination result. For $k\geq 1$, equality for graphs in $\mathcal{T}$ follows from  Proposition \ref{prop:twotermpathsortermcycle}. For the case of $K_{3,3}$, $\gp{K_{3,3}}=2$ and we will show in Corollary \ref{cor:K33} that $\gpk{K_{3,3}} = k +\left\lfloor \frac{k}{5}\right\rfloor +2 < 2(k+1)$ for $k\geq 1$. 
    \end{proof}

\subsection{The overlapping power dominating sets bound}

It is often the case in PMU-defect-robust power domination that we would like to overlap power dominating sets in such a way that the combined set is PMU-defect-robust. Therefore, the following parameter is introduced.

\begin{defn}\label{def:sj}
Let $G$ be a graph and choose an integer $j \geq \gp{G}$. Let $\bigpdsj{G}$ be the maximum cardinality of a set $B \subseteq V(G)$ such that if $A \subseteq B$ and $|A|=j$, then $A$ is a power dominating set. If no such set $B$ exists, $\bigpdsj{G}$ is undefined. When $j=\gp{G}$, we denote $\bigpdsjother{G}{\gp{G}}$ by $\bigpds{G}$.
\end{defn}

For examples of $\bigpds{G}$, see Figure \ref{fig:bigpds}.

\begin{figure}[htbp]
\begin{center}
\includegraphics[scale=1]{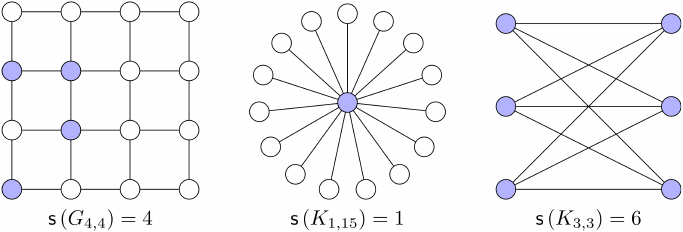}

\end{center}
    \captionsetup{width=.9\linewidth}
\caption{Examples of the parameter $\bigpds{G}$, with the set shaded}
\label{fig:bigpds}
\end{figure}

\begin{obs}
Note that a set $B$ which satisfies the conditions of Definition \ref{def:sj} is a $k$-rPDS and also a $k$-fault-tolerant set for $k =\bigpdsj{G} - j$. 
\end{obs}
The connection between $\bigpdsj{G}$, fault-tolerance, and robustness when $j=\gp{G}$, is shown in the next proposition.
\begin{prop}\label{prop:gpklowQ}
    For any graph $G$ and $k\geq 0$,  if  $\bigpds{G} \geq k+\gp{G}$ then $\vftk{G}=\gpk{G}=k+\gp{G}$.
    \end{prop}

    \begin{proof}
    If $\bigpds{G} \geq k+\gp{G}$, then there exists a set $S$ of size at least $k+\gp{G}$ so that any $\gp{G}$ elements of $S$ form a power dominating set of $G$. Thus, any $\gp{G}+k$ elements of $S$ form a $k$-rPDS and also a $k$--fault-tolerant power dominating set of $G$ of size $\gp{G}+k$ and so $\gpk{G}\leq \vftk{G} \leq \gp{G}+k$. By the lower bound in Proposition \ref{basebounds}, $\gpk{G} \geq \gp{G}+k$. 
    \end{proof}

Note that in Definition \ref{def:sj}, if $j>\gp{G}$, it is possible that some subsets of $B$ of size $j$ will not be minimal, however, these subsets will still be power dominating sets. Thus, we use the parameter $\bigpdsj{G}$ to create a new upper bound.

\begin{thm}\label{sjbound}
Given a graph $G$, integers $k \geq 1$ and $\bigpdsj{G} > j \geq \gp{G}$.  Using the division algorithm, determine non-negative integers $q$ and $r$ such that $k+1=q(\bigpdsj{G} -j+1)+r$ and $0 \leq r < \bigpdsj{G}-j+1$. Then, 
\[
\gpk{G} \leq 
\left\{\begin{array}{ll}
  \bigpdsj{G} \left(\frac{k+1}{\bigpdsj{G}-j+1}\right)         & \text{if $r=0$}\\
   \bigpdsj{G} \left\lfloor \frac{k+1}{\bigpdsj{G}-j+1} \right\rfloor + r+(j-1)         &\text{if $r> 0$}
    \end{array} \right.
\]
\end{thm}

\begin{proof}

Let $B= \{v_1,v_2,\ldots,v_{\bigpdsj{G}}\} \subseteq V(G)$ such that every subset of $B$ of order $j$ is a power dominating set. Use the division algorithm to find nonnegative integers $q$ and $r$ such that $k+1=q(\bigpdsj{G}-j+1)+r$ and $0 \leq r < \bigpdsj{G}-j+1$. 

If $r=0$, then place $q=\frac{k+1}{\bigpdsj{G}-j+1}$ PMUs on each vertex in $B$. Note that we need at least $j$ vertices to have PMUs after $k$ PMUs are removed. However, any  set of $\bigpdsj{G}-j+1$ vertices in $B$ has $k+1$ PMUs, and so the maximum number of vertices that we can remove all PMUs from is $\bigpdsj{G}-j$. This leaves at least $j$ vertices with PMUs, and so we have constructed a $k$-rPDS of $G$.

If $r>0$, let $q= \left \lfloor \frac{k+1}{\bigpdsj{G}-j+1} \right \rfloor$.
Place $q+1$ PMUs on the first $r+j-1$ vertices of $B$. Place $q$ PMUs on the remaining $(\bigpdsj{G}-j+1)-r$ vertices of $B$. Again, we need at least $j$ vertices to have PMUs after $k$ PMUs are removed. Observe that the $\bigpdsj{G}-j+1$ vertices of $B$ with the fewest PMUs have a total of 
\begin{align*}
    q((\bigpdsj{G}-j+1)-r) + (q+1)r  &= q(\bigpdsj{G}-j+1)+r\\
                                &=k+1
\end{align*}
PMUs. This means that in the worst case, we have $j-1$ vertices with $q+1$ PMUs and $1$ vertex with $1$ PMU, and so we have constructed a $k$-rPDS of $G$.
\end{proof}

We can apply Theorem \ref{sjbound} to improve the bound from Proposition \ref{basebounds} as demonstrated in the next example.

\begin{ex}
Consider the grid graph on 16 vertices, $G_{4,4}$, shown in Figure \ref{fig:bigpds}. Observe that $\gp{G_{4,4}}=2$ and $\bigpds{G_{4,4}}=4$. Then the bound in Theorem \ref{sjbound} reduces to
\[
\gpk{G_{4,4}} \leq 
\left\{\begin{array}{ll}
   4 \left(\frac{k+1}{3}\right)         & \text{if $r=0$}\\
   4 \left\lfloor \frac{k+1}{3} \right\rfloor + r+1         &\text{if $r> 0$}
    \end{array} \right.
\]
where $r\equiv k+1 \pmod 3$. Note that the bound  in Proposition \ref{basebounds} yields $\gpk{G_{4,4}}\leq 2(k+1)$. Therefore, the bound in Theorem \ref{sjbound} is an improvement.
\end{ex}

The bound in Theorem \ref{sjbound} is not necessarily tight when $\bigpds{G} < |V(G)|$. To see this, Example \ref{K34} in Section~\ref{sec:completebipartite} gives a formula for $\gpk{K_{3,4}}$ for all $k \geq 0$. Note that $\gpk{K_{3,4}}=2$ and $\bigpds{K_{3,4}}=4$. The bound in Theorem \ref{sjbound} agrees with the $k$-PMU-defect-robust power domination number when $k \leq 5$. When $k=6$ the upper bound given by Theorem \ref{sjbound} is $\gpk{G} \leq 10$ but the actual value is $\gpk{G}=9$. 

\subsection{The $\bigpdsj{G}$ bound when $j > \gp{G}$}

The natural question is whether  Theorem \ref{sjbound} is tight and if not, whether it improves the upper bound of Proposition \ref{basebounds}.

\begin{obs}\label{obs:sjboundnottight}
The $\bigpdsj{G}$ bound in Theorem \ref{sjbound} cannot be tight for all $k$. To see this, given a graph $G$, an integer $j>\gp{G}$ and $\bigpdsj{G}>j$,  suppose for an integer $k \geq 1$, $k+1$ is divisible by $\bigpdsj{G}-j+1$. Define $U=\bigpdsj{G}\left(\frac{k+1}{\bigpdsj{G}-j+1}\right)$, the bound given by Theorem \ref{sjbound} for a $k$-rPDS. Then Theorem \ref{sjbound} implies $\ddot{\gamma}_p^{k+1}(G) \leq \bigpdsj{G} \left\lfloor \frac{k+2}{\bigpdsj{G}-j+1} \right\rfloor + 1+(j-1) = U+j$, but Lemma \ref{plusgammap} gives an upper bound of $\ddot{\gamma}_p^{k+1}(G) \leq U +\gp{G}$, which is smaller.
\end{obs}

Although the $\bigpdsj{G}$ upper bound can never be tight for all $k \geq 1$, if $j > \gp{G}$, it is possible that it is tight for some $k \geq 1$. 

\begin{ex}\label{ex:gridgraphs}
Bjorkman \cite{beththesis} calculated the PMU-defect-robust power domination numbers for small square grid graphs, $G_{n,n}$, for small $k$, shown in Table \ref{tab:G66}.  Dorfling and Henning \cite{dh06} showed that $\gp{G_{6,6}}=2$ and Bjorkman \cite{beththesis} showed that $\bigpds{G_{6,6}}=2$ since no set of 3 vertices has the property that any subset of 2 vertices is a power dominating set. However, Bjorkman found $\{(2,4),(3,2),(6,5),(4,4),(2,6)\}$, which is a set of 5 vertices for which any 3 out of the 5 vertices is a power dominating set, and there is no set of 6 vertices for which every subset of 3 vertices is a power dominating set. Therefore, $\bigpdsjother{G_{6,6}}{3}=5$. The bound in Theorem \ref{sjbound} improve the basic upper bound in Proposition \ref{basebounds} for all $k \geq 4$, as shown in Table \ref{tab:G66}.
\end{ex}

    \begin{center}
    \begin{table}[htbp]
    \centering
    \captionsetup{width=.9\linewidth}
    \caption{Comparison of upper bounds for $G_{6,6}$. The $\bigpdsj{G}$ bound is from Theorem \ref{sjbound} and the $(k+1)\gp{G}$ bound is from Proposition \ref{basebounds}.}\label{tab:G66}
    \begin{tabular}{cccc}
    \hline
    $k$ & $\gpk{G_{6,6}}$ & $\bigpdsj{G}$ bound & $(k+1)\gp{G}$ \\
    \hline
    \hline
    0 & 2 & N/A & 2\\
    \hline
    1 & 4 & N/A  &  4\\ 
    \hline
    2 & 5 &  5 & 6\\
    \hline
    3 & 7 &8 & 8\\
    \hline
    4 & 8 or 9 & 9 & 10\\
    \hline
    5 & 9 or 10 & 10 & 12\\
    \hline
    6 & & 13 & 14\\
    \hline
    7 &  & 14 & 16\\
    \hline
    8 &  & 15  & 18 \\
    \hline
    \end{tabular}
    \end{table}
    \end{center}

 In Table \ref{tab:G66}, we see that as $k$ increases, the gap between these bounds grows. We formalize this in the following proposition.
 
 \begin{prop}\label{jbetterif}
Let $G$ be a graph so that for some integer $j > \gp{G}\geq 2$ such that $\bigpdsj{G} > j$ and 
\[\gp{G} > \frac{\bigpdsj{G}}{\bigpdsj{G}-j+1}.\]
Then the upper bound in Theorem \ref{sjbound} improves the upper bound in Proposition \ref{basebounds} for sufficiently large $k$.
\end{prop}

\begin{proof}
Using the division algorithm, determine non-negative integers $q$ and $r$ such that $k+1=q(\bigpdsj{G} -j+1)+r$ and $0 \leq r < \bigpdsj{G}-j+1$.  Consider the case in which $k+1$ is divisible by $\bigpdsj{G}-j+1$, that is, $r=0$. The difference between the two bounds is
\[
(k+1)\gp{G}-\frac{\bigpdsj{G}(k+1)}{\bigpdsj{G}-j+1}=(k+1)\left(\gp{G}- \frac{\bigpdsj{G}}{\bigpdsj{G}-j+1}\right).
\]
The upper bound in Theorem \ref{sjbound} improves the basic bound when this number is positive. In other words,
\[
\gp{G} > \frac{\bigpdsj{G}}{\bigpdsj{G}-j+1}.
\]

When $k+1$ is not evenly divisible by $\bigpdsj{G}-j+1$, that is, when $r>0$, the difference between the two bounds is 
\begin{align*}
(k+1)\gp{G} - & \left(\bigpdsj{G} \left\lfloor \frac{k+1}{\bigpdsj{G}-j+1} \right\rfloor + r+(j-1) \right) \\ 
            &> (k+1)\gp{G} - \bigpdsj{G} \left( \frac{k+1}{\bigpdsj{G}-j+1}\right) - r-j+1 \\
            &= (k+1)\left(\gp{G}-\frac{\bigpdsj{G}}{\bigpdsj{G}-j+1}\right) - (r+j-1)\\
            &> (k+1)\left(\gp{G}-\frac{\bigpdsj{G}}{\bigpdsj{G}-j+1}\right) - \bigpdsj{G}.
\end{align*}

Observe that the first term above is a multiple of $k+1$ and the second term is constant. Thus, if $j$ and $\bigpdsj{G}$ satisfy the same conditions as in the first case, the difference will be positive for sufficiently large values of $k$.
\end{proof}

\subsection{When any set of size $\gp{G}$ is a power dominating set}
When $\bigpds{G}=|V(G)|$, any set of size $\gp{G}$ is a power dominating set of $G$. In this case, the bound in Theorem \ref{sjbound} is tight.

\begin{prop}\label{sequalsn}
Given a graph $G$ on $n$ vertices, an integer $k>0$, and $\bigpds{G}=|V(G)|$. Using the division algorithm, determine non-negative integers $q$ and $r$ such that $k+1=q(n-\gp{G}+1)+r$ and $0 \leq r < n-\gp{G}+1$. Then,
\[
\gpk{G} =
\left\{\begin{array}{ll}
  n \left(\frac{k+1}{n-\gp{G}+1}\right)         & \text{if $r=0$}\\
   n \left\lfloor \frac{k+1}{n-\gp{G}+1} \right\rfloor + r+(\gp{G}-1)         &\text{if $r> 0$}
    \end{array} \right.
\]
\end{prop}

\begin{proof} By Theorem \ref{sjbound}, this is an upper bound.

To show the lower bound, consider a $\gpkset$, $S$, of $G$ and label the vertices of $G$ as $v_1,v_2,\dots,v_n$ such that $\pmus{v_i} \leq \pmus{v_j}$ if $i < j$. Then, 
\[
\sum_{i=1}^{n-\gp{G}+1} \pmus{v_i} \geq k+1
\]
since $k$ PMUs may be removed from the vertices with the fewest PMUs, leaving at least $\gp{G}$ vertices with at least one PMU each. By the Pigeonhole Principle, if there are at least $k+1$ PMUs on $n-\gp{G}+1$ vertices, at least one of these vertices $w$ must have multiplicity $\pmus{w} \geq \frac{k+1}{n-\gp{G}+1}$. This implies that the $\gp{G}-1$ vertices of largest multiplicity must have at least $q$ PMUs each if $r=0$ or at least $q+1$ PMUs each if $r>0$. This implies that $\gpk{G}$ is greater than or equal to the given bound, and hence equality holds.
\end{proof}

We can use Proposition~\ref{sequalsn} to determine the $k$-PMU-defect-robust power domination number of any graph with $\bigpds{G}=|V(G)|$. A \emph{complete multipartite graph} with $m$ partite sets, $K_{a_1,a_2,\ldots,a_m}$, has vertex set $V(K_{a_1,a_2,\ldots,a_m}) = A_1 \cup A_2 \cup \ldots \cup A_m$ and edge set $E(K_{a_1,a_2,\ldots,a_m}) = \{v_iv_j : v_i\in A_i, v_j\in A_j, i\neq j\}$. The graph $K_{3,3,3}$ is shown in Figure \ref{fig:multipartite}.

\begin{figure}[htbp]
\begin{center}
\includegraphics[scale=1]{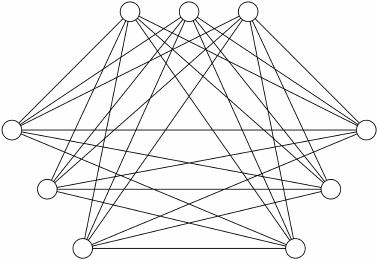}
\end{center}
    \captionsetup{width=.9\linewidth}
\caption{The graph $K_{3,3,3}$}
\label{fig:multipartite}
\end{figure}

\begin{prop}\label{multipartite}
Let $G$ be the complete multipartite graph $K_{3,3,\dots,3}$ with $m$ partite sets for any $m \geq 2$. Then $\bigpds{G}=|V(G)|=3m$. 
\end{prop}

\begin{proof}
Given $m \geq 2$, consider $K_{3,3,\dots,3}$ with $m$ partite sets. Observe that $\gp{G}=2$. Consider any two vertices $v,w \in V(K_{3,3,\dots,3})$. If $v$ and $w$ are in different partite sets, then $N(v) \cup N(w)=V(K_{3,3,\dots,3})$, so the set $\{v,w\}$ is a power dominating set since it is a dominating set. If $v$ and $w$ are in the same partite set $P$, the domination step observes all the vertices except the third vertex in $P$, which can be observed by a zero forcing step. 
\end{proof}

We can now apply Proposition \ref{sequalsn} to $K_{3,3,\dots,3}$.

\begin{cor}\label{multipartitekrob}
Let $G$ be the complete multipartite graph $K_{3,3,\dots,3}$ with $m$ partite sets for any $m \geq 2$. Then 
\[\gpk{K_{3,3,\dots,3}} = \left\{\begin{array}{ll}
  3m \left(\frac{k+1}{3m-1}\right)         & \text{if $r=0$}\\
   3m \left\lfloor \frac{k+1}{3m-1} \right\rfloor + r+(\gp{G}-1)         &\text{if $r> 0$}
    \end{array} \right.. \]
\end{cor}

Figure \ref{fig:sequalsnex} gives examples of other graphs $G$ for which $\gp{G}=2$ and $\bigpds{G}=|V(G)|$, which also have $\gpk{G}$ given by Proposition \ref{sequalsn}.

\begin{figure}[hbtp]
\begin{center}
\includegraphics[scale=1]{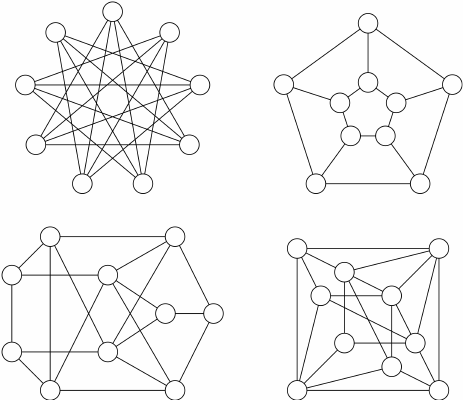}

\end{center}
    \captionsetup{width=.9\linewidth}
\caption{Examples of graphs $G$ for which $\gp{G}=2$ and $\bigpds{G}=|V(G)|$}
\label{fig:sequalsnex}
\end{figure}

A basic characteristic of graphs with $\bigpds{G}=|V(G)|$ is the following restriction on the minimum degree.

\begin{lem}\label{bigpdsNimplies}
Let $G$ be a connected graph on $n \geq 6$ vertices with $\gp{G} \geq 2$ and $\bigpds{G}=n$. Then $\delta(G) \geq 3$. 
\end{lem}

\begin{proof}
Since $\gp{G} \geq 2$, $G$ is not a path or a cycle, so $G$ has a vertex of degree $3$ or more. Assume for contradiction that there exists $x\in V(G)$ such that $\deg{x}{}=1$. Then there exists a minimum power dominating set of $G$ that contains both $x$ and its one neighbor, however, this set is not minimum as $x$ is redundant.
Assume for contradiction that there exists $v \in V(G)$ be such that $\deg{v}{}=2$. Let $w$ be the nearest high degree vertex on the path containing $v$. Then there is a set containing both $v$ and $w$ that is a minimum power dominating set. But any set containing $w$ will observe $v$, and then $v$ will observe its remaining neighbor in the zero forcing step, so $S$ is not minimum, a contradiction.
\end{proof}

\section{Complete Bipartite Graphs}\label{sec:completebipartite}

The complete bipartite graph $G=K_{a,b}$ is the graph with vertex set $V(G)=A \cup B$ with $A = \{ x_1, x_2, \dots, x_a \}$ and $B = \{ y_1, y_2, \dots, y_b \}$ and edge set $E(G)=\{ \{x_i,y_j\}:x_i \in A, y_j \in B\}$. If one of the partite sides contains one or two vertices, then $\gp{K_{a,b}}=1$ and so the bounds in Proposition \ref{basebounds} agree, giving $\gpk{K_{a,b}}=k+1$. When $3 \leq a \leq b$, the standard approach to finding a minimum set with a specific property is to fix $a$ and look for a pattern for $K_{a,b}$ for all $b \geq a$. Preliminary computation showed that no clear pattern emerged because the modular arithmetic depends on the size of both $a$ and $b$. Therefore, our approach is to find the values of $\gpk{G}$ for $k \geq 0$ recursively for an arbitrary $K_{a,b}$ with fixed $3 \leq a \leq b$. At the end of this section, we give examples of how to derive $\gpk{K_{a,b}}$ for specific values of $a$ and $b$ using this method.

\subsection{Essential multisets and the uncover number}

It is well known that $\gp{K_{a,b}}=2$ with minimum power dominating sets consisting of any one vertex from each side of the partition. A minimal, but not minimum, power dominating set is either a set of any $a-1$ vertices in $A$ or a set of any $b-1$ vertices in $B$. 
This means that for the complete bipartite graphs with $3\leq a \leq b$, essential multisets have 3 distinct structures based on the structure of the underlying minimal power dominating sets.

\begin{defn}\label{coverage}
For complete bipartite graphs $K_{a,b}$ with $3\leq a \leq b$, we distinguish the types of essential multisets contained in $S$ using the following notation:
\begin{enumerate}
\item \emph{$P$-essential multiset}: There is exactly one vertex in $A$ and one vertex in $B$ that have PMUs, that is, there exists $x_i\in A$ and $y_j\in B$ such that $\pmus{x_i}=\pmus{A}\geq 1$ and $\pmus{y_j}=\pmus{B}\geq 1$.
\item \emph{$A$-essential multiset}: No vertices in $B$ have a PMU and all except one vertex in $A$ has a PMU, that is, $|\{x_i \in A:\pmus{x_i} > 0\}| = a-1$.
\item \emph{$B$-essential multiset}: No vertices in $A$ have a PMU, and all except one vertex in $B$ has a PMU, that is, $|\{y_i \in B:\pmus{y_i} > 0\}| = b-1$.
\end{enumerate}
\end{defn}

We demonstrate $P$-essential, $A$-essential, and $B$-essential multisets of $K_{3,4}$ in Figure \ref{fig:K34minimalsets}.

 \begin{figure}[htbp]
 \begin{center}
 \includegraphics[scale=1]{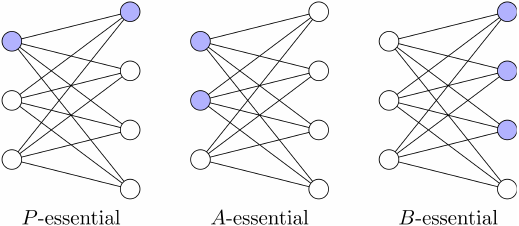}
 \end{center}
    \captionsetup{width=.9\linewidth}
 \caption{Underlying sets for each of the types of essential multisets of $K_{3,4}$}
 \label{fig:K34minimalsets}
 \end{figure}

Throughout this section, it will be useful to order the vertices by the number of PMUs they contain. For the graph $G=K_{a,b}$ with vertex set $V(G)=A \cup B$ where $A = \{ x_1, x_2, \dots, x_a \}$ and $B = \{ y_1, y_2, \dots, y_b \}$, and a given PMU placement, without loss of generality, label the vertices so that
\[
\pmus{x_1} \leq \pmus{x_2} \leq \dots \leq \pmus{x_a}
\]
and 
\[
\pmus{y_1} \leq \pmus{y_2} \leq \dots \leq \pmus{y_b}.
\]

A $k$-rPDS $S$ contains at least $k+2$ PMUs by Proposition \ref{basebounds}, so removing $k$ PMUs leaves at least 2 PMUS. If there are PMUs remaining on both sides of the partition after removing $k$ PMUs, then the resulting multiset remains a power dominating set since it contains a $P$-essential multiset. If the removal of $k$ PMUs leaves one side without a PMU, there must be only one unobserved vertex on the other side to power dominate. Thus, it is useful to define a parameter that quantifies the maximum number of PMUs that may be removed from side A (side B) and still remain an $A$-essential ($B$-essential) multiset.

\begin{defn}\label{uncover}
Let $G=K_{a,b}$ with $3 \leq a \leq b$ and assume $V(G)=A \cup B$ with $|A|=a$ and $|B|=b$. Let $S$ be a multiset of the vertices. Let the vertices of $A$ be labelled so that $\pmus{x_1} \leq \pmus{x_2} \leq \dots \leq \pmus{x_a}$ and similarly label the vertices of $B$ so that $\pmus{y_1} \leq \pmus{y_2} \leq \dots \leq \pmus{y_b}$. The \emph{uncover number for side $A$}, denoted $u_a$, is defined to be \[u_a=\pmus{x_1}+\pmus{x_2}-1.\] Similarly, define the \emph{uncover number for side $B$}, denoted $u_b$, to be \[u_b=\pmus{y_1}+\pmus{y_2}-1.\]
\end{defn}

We make the following observation about negative uncover numbers.

\begin{obs}
It is possible that the uncover number is negative, since $u_a=-1$ when $\pmus{x_1}=\pmus{x_2}=0$. This represents the case that the original multiset $S$ does not contain an $A$-essential multiset. Similarly, $u_b=-1$ when $S$ does not contain a $B$-essential multiset.
\end{obs}

The uncover number is useful in determining how many PMUs can be removed from each partite set, as demonstrated in the following lemma.

\begin{lem}\label{maxremove}
Given a multiset $S$ with $u_a \geq 0$ ($u_b \geq 0$), then $u_a$ ($u_b$) is the maximum number of PMUs that can be removed from side A (side B) with the resulting multiset being an $A$-essential ($B$-essential) multiset.
\end{lem}

\begin{proof}
Assume first that $u_a=0$. This implies $\pmus{x_1}+\pmus{x_2}=1$, and $\pmus{x_j} \geq 1$ for all $2 \leq j \leq a$. However, the removal of one PMU from $x_2$ results in a multiset that is not $A$-essential. Therefore, the maximum number of PMUs that can be removed from side $A$ to ensure that the resulting multiset is $A$-essential is 0.

If $u_a >0$, $\pmus{x_1}+\pmus{x_2}=u_a+1$, so the removal of $u_a +1$ PMUs may be chosen so that both $x_1$ and $x_2$ have no PMUs, and so the resulting multiset is not $A$-essential. On the other hand, since $\pmus{x_i}+\pmus{x_j} \geq u_a$ for all $1 \leq i < j \leq a$, the removal of $u_a$ PMUs can result in at most one vertex having no PMUs, and so the resulting multiset is $A$-essential. Therefore, the maximum number of PMUs that can be removed from side $A$ to ensure that the resulting multiset is $A$-essential is $u_a$.

The same argument holds for side B.
\end{proof}

We can extend Lemma \ref{maxremove} to determine how many PMUs must be on $A$ and $B$ for a $k$-rPDS.

\begin{thm}\label{kormore}
The multiset $S$ of PMUs for $G=K_{a,b}$ when $3 \leq a \leq b$, $\pmus A =x$ and $\pmus B = y$ and the uncover numbers $u_a$ and $u_b$ defined as above is a $k$-rPDS if and only if $x+u_b \geq k$ and $y+u_a \geq k$. 
\end{thm}

\begin{proof}
To see why this condition is necessary, note that if $x+u_b < k$, then it is possible to remove $k$ PMUs so that the resulting multiset is no longer a $B$-essential multiset by Lemma \ref{maxremove} and all PMUs are on $B$, meaning that we have found a submultiset that does not contain an essential multiset in violation of Observation \ref{obs:krPDScontainsessential}. Similarly, if $y+u_a<k$, removal of $k$ PMUs results in a multiset that is no longer $A$-essential and can contain no other essential multiset.

For the sufficiency, assume $x+u_b \geq k$ and $y+u_a \geq k$. Note that $S$ must have at least $k+2$ PMUs by Proposition \ref{basebounds}. First consider the case that the uncover numbers are negative. If $u_a=u_b=-1$, then $x+u_b \geq k$ yields $x \geq k+1$. Similarly, $y \geq k+1$. In this case removal of any $k$ PMUs always results in a $P$-essential multiset. If $u_a=-1$ and $u_b \geq 0$, the multiset $S$ contains a $B$-essential multiset, and possibly a $P$-essential multiset if $x>0$. Note that $y+u_a \geq k$ yields $y \geq k+1$. Since $x+u_b \geq k$, removal of $k$ PMUs will always leave at least one PMU on side B, and if all PMUs are removed from side A, $x+u_b\geq k$ ensures that at most one vertex will be have no PMUs on side B, that is, a $B$-essential multiset. In any case, we find that the removal of any $k$ PMUs results in a multiset that contains an essential multiset and so $S$ is a $k$-rPDS.
 
When the uncover numbers are both nonnegative, choose a multiset $F$ of $k$ PMUs to remove from $S$. If there exists at least one PMU remaining on each side of the partition, the resulting multiset is $P$-essential. If all of the PMUs are removed from side A, $k-x$ PMUs must be removed from side B. However, $k-x \leq u_b$, so the multiset $S-F$ is a $B$-essential multiset. Similarly, if all the PMUs are removed from side B, $k-y$ PMUs need to be removed from side A, and the multiset $S-F$ is $A$-essential. In all cases, $S-F$ contains an essential multiset, and so $S$ is a $k$-rPDS.
\end{proof}

Multisets of vertices of $K_{a,b}$ with fixed $\pmus{A}$ and $\pmus{B}$ may have wildly different distributions of the PMUs among the vertices on each side. The maximum number of PMUs that may be removed from a side and maintain an $A$-essential (or $B$-essential) multiset occurs when the two vertices with the smallest number of PMUs contain as many as possible. That is, we want to maximize $u_a$ and $u_b$ to maximize how many PMUs can be removed. One optimal way to distribute the PMUs on each side, both to maximize the uncover number and to simplify the notation in the ensuing argument, is to distribute the given number of PMUs evenly. Note that this definition does not require $\pmus{A}=\pmus{B}$.

\begin{defn}\label{def:evenlydistributed}
A multiset $S$ of vertices of $K_{a,b}$ is called {\em evenly distributed} if the the PMUs on each side of the partition are distributed evenly among the vertices. Formally, suppose $\pmus{A} = x$ and $\pmus{B}= y$. Using the division algorithm, there exists unique nonnegative integers $q_a$, $r_a$, $q_b$ and $r_b$ such that $x=q_aa+r_a$, $0 \leq r_a<a$ and $y=q_bb+r_b$, $0 \leq r_b < b$. The multiset $S$ is \emph{evenly distributed} if $\pmus{x_i} \in \{q_a,q_a+1\}$ for all $x_i \in A$ and $\pmus{y_j} \in \{q_b,q_b+1\}$ for all $y_j \in B$.
\end{defn}

The following lemma shows that only evenly distributed multisets need to be examined when searching for the minimum cardinality of a $k$-rPDS of $K_{a,b}$.

\begin{lem}\label{evenworks}
For a complete bipartite graph $G=K_{a,b}$ with $3 \leq a \leq b$ an integer $k>0$ and a $k$-rPDS $S$ of order $t$, there exists an evenly distributed multiset $S^e$ with the same $\pmus{A}$ and $\pmus{B}$ such that $S^e$ is also a $k$-rPDS.
\end{lem}

\begin{proof}
Let $G=K_{a,b}$ with vertices partitioned into sets $A$ and $B$ with $|A|=a$ and $|B|=b$, $S$ with $|S|=t$, $\pmus{A} = x$ and $\pmus{B}= y$ with $x+y=t$. Without loss of generality, label the vertices on each side of the graph in increasing order of the number of PMUs on each vertex in $S$. Define $S^e$ to be the evenly distributed multiset with $x$ PMUs on side $A$ and $y$ PMUs on side $B$. 

If $x \geq k+1$ and $y \geq k+1$ the distribution of the PMUs does not matter, since removing any multiset of size $k$ will result in a multiset containing a $P$-essential multiset, so $S^e$ is also a $k$-rPDS.

If $x \leq k$ or $y \leq k$, $S$ must contain an $A$-essential or a $B$-essential multiset. (If both sides contain $k$ or fewer PMUs, $S$ must contain both an $A$-essential and a $B$-essential multiset.) 

There exist nonnegative integers $q_a$, $r_a$, $q_b$ and $r_b$ such that $x=q_aa+r_a$, $0 \leq r_a<a$ and $y=q_bb+r_b$, $0 \leq r_b < b$. In the evenly distributed multiset $S^e$, the maximum number of PMUs that may be removed from $X$ so that the resultant set still contains an $A$-essential multiset is 
\[
    u_a^e =
    \left\{\begin{array}{ll}
    2q_a    & \text{if $r_a=a-1$}\\
    2q_a-1  & \text{if $r_a < a-1$}
    \end{array} \right.
\]
Similarly,
\[
    u_b^e =
    \left\{\begin{array}{ll}
    2q_b            & \text{if $r_b=b-1$}\\
    2q_b-1          &\text{if $r_b < b-1$}
    \end{array} \right.
\]
 We will now show that $u_a \leq u_a^e$ and $u_b \leq u_b^e$. First consider the case that $r_a<a-1$. Suppose for contradiction that $u_a > u_a^e$. By the definition of the uncover number, 
 \[\pmusset{x_1}{S} +\pmusset{x_2}{S}-1 =u_a > u_a^e = 2q_a-1\] 
 and so $\pmusset{x_1}{S}+\pmusset{x_2}{S} > 2q_a$. Then by the Pigeonhole Principle and the ordering of the vertices, $\pmusset{x_2}{S} > q_a$ and so each of the other vertices $x_i\in X$ for $i>2$ has $\pmusset{x_i}{S} \geq q_a+1$. But then 
 \[\pmus{A} = x \geq 2q_a+1 +(a-2)(q_a+1) \geq aq_a+(a+1)>x,\]
 a contradiction. A similar argument holds for the case that $r_a=a-1$, and a symmetric argument also shows $u_b \leq u_b^e$.  Therefore, $y+u_a^e \geq k$ and $x+u_b^e \geq k$. Thus $S^e$ is a $k$-rPDS by Theorem \ref{kormore}.
\end{proof}

\subsection{The Parameter Sequence}

Let $G=K_{a,b}$ for a specific $3 \leq a \leq b$. For all $k \geq 0$, define $(p_k)$ to be the sequence of positive integers so that $p_k=\gpk{G}$. In this section, we will construct this sequence recursively from the base case that $p_0=\gp{K_{a,b}}=2$ for any complete bipartite graph with $3\leq a \leq b$. Proposition \ref{prop:incr} and Lemma \ref{plusgammap} combine to yield the following proposition.

\begin{prop}\label{OneorTwo}
Given the complete graph $K_{a,b}$ with $3 \leq a \leq b$ and an integer $k \geq 0$, \[\gpk{K_{a,b}} + 1 \leq \gpkplus{K_{a,b}} \leq \gpk{K_{a,b}} + 2. \]
\end{prop}

By Proposition \ref{OneorTwo}, $p_{k+1}=p_k+1$ or $p_{k+1}=p_k+2$. This leads to the main focus of this section:

\begin{quest}\label{quest:OneorTwo}
For a complete bipartite graph $K_{a,b}$ with $3\leq a \leq b$, when is the addition of two PMUs necessary when moving from a $k$-rPDS to a $k+1$-rPDS?
\end{quest}

In order to answer Question \ref{quest:OneorTwo}, we will need the following definition.

\begin{defn}\label{AddTwoCond}
Given integers $3 \leq a \leq b$, $k \geq 0$, and $p_k = \gpk{K_{a,b}}$, the pair $(k,p_k)$ is said to satisfy the {\em Jump Condition} for $K_{a,b}$ if there exists nonnegative integers $q_a$, $r_a$, $q_b$ and $r_b$ for which $p_k=q_aa+r_a + q_bb+r_b$  such that the following relationships hold:
\begin{enumerate}
    \item $q_aa+r_a+2q_b-1=k$
    \item $q_bb+r_b+2q_a-1=k$
    \item $r_a \leq a-3$
    \item $r_b \leq b-3$
\end{enumerate}
\end{defn}

Observe that the requirement that $q_aa+r_a + q_bb+r_b = p_k$ defines a minimum evenly distributed $k$-rPDS of $K_{a,b}$, denoted $S^e$, with $q_aa+r_a$ PMUs on side A, and $q_bb+r_b$ PMUs on side B. We examine this set in Lemma \ref{uniqueK}.

\begin{lem}\label{uniqueK}
If the pair $(k, p_k)$ satisfies the Jump Condition, then the evenly distributed multiset $S^e$ with $\pmus{A}=q_aa+r_a$ and $\pmus{B}=q_bb+r_b$ is the unique evenly distributed minimum $k$-rPDS.
\end{lem}

\begin{proof}

The multiset $S^e$ has $x=q_aa+r_a$ PMUs on side A, and $y=q_bb+r_b$ PMUs on side B. To show that $S^e$ is the unique evenly distributed minimum $k$-rPDS, it necessary to show that no other multiset $C$ with $|C|=p_k$ and a different number of PMUs on side A and side B is a $k$-rPDS. Let $C^w$ be an evenly distributed multiset with $x-w=q_aa+r_a-w$ PMUs on side A, and $y+w=q_bb+r_b+w$ PMUs on side B, and let $u_a^w$ and $u_b^w$ be the respective uncover numbers for $C^w$.

Consider the case when $w=1$, that is, the multiset $C^1$ with 1 fewer PMU on side $A$ and 1 more PMU on side $B$. Then the condition that $r_b \leq b-3$ implies that there are at least 3 vertices so that $\pmusset{y_1}{S^e}=\pmusset{y_2}{S^e}=\pmusset{y_3}{S^e}= q_b$, and so by moving only 1 PMU to side $B$, $\pmusset{y_1}{S^e}=\pmusset{y_2}{S^e}=\pmusset{y_1}{C^1}=\pmusset{y_2}{C^1}$ and so $u_b^1=u_b$. However, $\pmusset{A}{C^1}+u_b^1 = (x-1) + u_b = k-1$, so $C^1$ is not a $k$-rPDS by Theorem \ref{kormore}. 

The maximum number of PMUs that can be removed from $C^w$ and retain a $B$-essential multiset is $x-w+u_b^w$, and this function is a nonincreasing function of $w$. To see this, note that if $x-w+u_b^w=m_w$, then removing one more PMU from side A and adding it to side B can only increase $u_b$ by at most one. Thus $m_{w+1}=x-(w+1)+u_b^{w+1} \leq m_w$. Since $m_1=k-1$, shifting more PMUs to side B results in a multiset that is not a $k$-rPDS. A symmetric argument shows that shifting PMUs to side A results in a multiset $C^{-w}$ that is not a $k$-rPDS. Thus, $S^e$ is the unique evenly distributed multiset which is a minimum $k$-rPDS.
\end{proof}

As an implication of Lemma \ref{uniqueK}, any multiset $S$ that is a minimum $k$-PDS for $K_{a,b}$ when $(k,p_k)$ satisfies the Jump Condition must have $\pmus{A}=q_aa+r_a$, $\pmus{B}=q_bb+r_b$, $u_a=2q_a-1$, and $u_b=2q_b-1$ \emph{even if the PMUs are not distributed evenly}. 

We are now able to answer Question \ref{quest:OneorTwo} in Theorem \ref{Jumpadds2}.

\begin{thm}\label{Jumpadds2}
Given $G=K_{a,b}$ with $3 \leq a \leq b$, $k \geq 0$ and $p_k=\gpk{G}$, then
\[
    p_{k+1}= 
    \left\{\begin{array}{ll}
    p_k+2            & \text{if $(k,p_k)$ satisfies the Jump Condition}\\
    p_k+1 &\text{otherwise}.
    \end{array} \right.
\]
\end{thm}

\begin{proof}
First assume that $(k,p_k)$ does not satisfy the Jump Condition. Let $S$ be a minimum evenly distributed $k$-rPDS with $x=q_aa+r_a$ vertices on side A, and $y=q_bb+r_b$ vertices on side B. 

In the case that the Jump Condition fails due to part 1 or 2, without loss of generality, assume $y+u_a \neq k$. Since $S$ is a $k$-rPDS, Theorem \ref{kormore} implies $y+u_a \geq k$ and $x+b_b\geq k$. Then we obtain $y+u_a\geq k+1$. Add one PMU to side A to create the multiset $S^{+1}$. Then $x+1+u_b \geq k+1$. By Theorem \ref{kormore}, this implies $S^{+1}$ is $(k+1)$-robust, and thus $p_{k+1}=p_k+1$. 

In the case that the Jump Condition fails due to part 3 or 4, without loss of generality, assume $a-2 \leq r_a < a$. Create the multiset $S^{+1}$ by adding one PMU to a vertex in $A$ with $q_a$ PMUs. Then $u_a^{+1}=u_a+1$ and so $y+u_a^{+1} \geq k+1$. Also, $(x+1) + u_b \geq k+1$, implying $S^{+1}$ is a $(k+1)$-rPDS of order $p_k+1$.

If $(k,p_k)$ satisfies the Jump Condition, let $S^e$ be the unique evenly distributed minimum $k$-rPDS defined by the Jump Condition as shown in Lemma \ref{uniqueK}. Suppose for eventual contradiction that $T$ is a minimum $(k+1)$-rPDS of order $p_k+1$ with $x$ PMUs on side A and $y$ PMUs on side B. By Lemma \ref{evenworks}, there exists an evenly distributed $(k+1)$-rPDS, $T^e$, with the same number of PMUs on each side. Remove one PMU from $T^e$ in such a way as the resulting multiset $T^-$ remains evenly distributed. Since $T^e$ is a $k+1$-rPDS, $T^-$ must be a $k$-rPDS by Observation \ref{removeone}. Then as $T^-$ is evenly distributed, $S^e$ is evenly distributed, and $(k,p_k)$ satisfies the Jump Condition, by Lemma \ref{uniqueK}, $T^-$=$S^e$. Adding one vertex to either side of $S^e$ does not change the uncover numbers since parts 3 and 4 of the Jump Condition imply there are at least 3 vertices with $q_a$ or $q_b$ PMUs. If the extra vertex is added to side A, then $y+u_a=k$ and $x+1+u_b=k+1$, and so the multiset $T^e$ is not $k+1$-robust, contradicting the assumption. A similar result holds if the extra PMU is added to side B. Therefore, a $(k+1)$-rPDS of order $p_k+1$ is impossible, and $p_{k+1}=p_k+2$.
\end{proof}

\subsection{Examples of the Sequence Construction}
\subsubsection{Balanced Bipartite Graphs}

The complete bipartite graph $K_{a,b}$ is called \emph{balanced} if $a=b$. In this section a closed formula for the sequence in Theorem \ref{Jumpadds2} is calculated for balanced complete bipartite graphs.

\begin{thm}\label{thm:Knn}
Given the complete bipartite graph $K_{n,n}$ with $n \geq 3$, 
\[
\gpk{K_{n,n}} =
            \begin{cases}
                2(k+1) -4 \left\lfloor \frac{k}{n+2} \right\rfloor & k \equiv 0,1,\dots,n-3 \pmod{n+2} \\
                k+2 +(n-2) \left\lfloor \frac{k}{n+2}\right\rfloor  & \text{otherwise}
            \end{cases}
\]
\end{thm}

\begin{proof}
Given an integer $n \geq 3$, the Jump Condition for the graph $K_{n,n}$ requires that integers $q_a$, $r_a$, $q_b$ and $r_b$ be found such that 
\begin{enumerate}
\item $q_a \cdot n +r_a + 2q_b -1 =k$
\item $q_b \cdot n + r_b + 2q_a -1 = k$
\item $0 \leq r_a \leq n-3$ and $0 \leq r_b \leq n-3$.
\end{enumerate}
Combining Part 1 and Part 2 and simplifying yields
\[ (n-2)(q_a - q_b) = r_b - r_a.\]
Since $n-2$ is a factor of $r_b - r_a \leq n-3$, this implies $r_b-r_a=0$ and so $q_a-q_b = 0$. Substituting $q=q_a=q_b$ and $r=r_a=r_b$, the Jump Condition requires 
\[ k = (n+2)q +r-1. \]
For $k \equiv -1, 0, 1, \dots, n-4 \pmod{n+2}$, increasing from a $k$-rPDS to a $(k+1)$-rPDS requires 2 extra PMUs. In other words, an additional PMU is needed for $k \equiv 0, 1, \dots, n-3  \pmod{n+2}$. The formula follows from counting the number of times an extra PMU are not needed in the first case, or counting how many extra PMUs are needed in the second case.
\end{proof}

\begin{cor}\label{cor:K33}
When $n = 3$, $k \equiv 0 \pmod{5}$ implies $k -4 \frac{k}{5} = \frac{k}{5}$, so Theorem \ref{thm:Knn} simplifies to 
\[ \gpk{K_{3,3}}=k+2+\left\lfloor\frac{k}{5}\right\rfloor. \]
\end{cor} 

We demonstrate $k$-rPDSs of $K_{3,3}$ for small values of $k$ in Figure \ref{fig:k33}.

 \begin{figure}[hbtp]
 \begin{center}
 \includegraphics[scale=1]{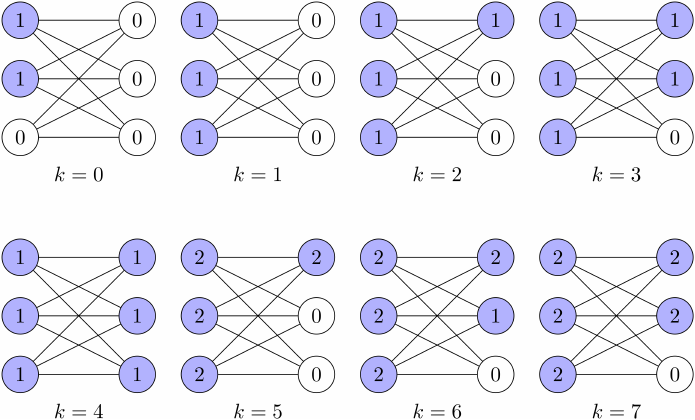}
 \end{center}
    \captionsetup{width=.9\linewidth}
 \caption{Minimum $k$-rPDS for $K_{3,3}$ for $k=0,1,\ldots,7$}
 \label{fig:k33}
 \end{figure}
 
\subsubsection{An Unbalanced Bipartite Graph Example}

When the complete bipartite graph is not balanced, Theorem \ref{Jumpadds2} must be applied to a specific graphs. The following is an example for small values of $a$ and $b$.

\begin{ex}\label{K34}
Let $G=K_{3,4}$. The Jump Condition for a specific $(k, p_k)$ requires that for $q_aa+r_a+q_bb+r_b=p_k$, the integers also satisfy
\begin{enumerate}
\item $3q_a+r_a+2q_b-1=4q_b+r_b+2q_a-1$
\item $3q_a+r_a+2q_b-1=k$
\item $r_a=0$
\item $0 \leq r_b \leq 1$
\end{enumerate}
Simplifying part 1 with $r_a=0$ yields $q_a=2q_b+r_b$. Substituting this for $q_a$ into the part 2 yields to $8q_b+3r_b-1=k$. When $r_b=0$, $k=8q_b-1$. When $r_b=1$, we obtain $8q_b+3-1= 8q_b+2=k$. Thus $\gpk{K_{3,4}}$ requires $k+2$ PMUs by Proposition \ref{basebounds} and an additional PMU for each $k \equiv 0,3 \pmod 8$. A closed form expression for $\gpk{G}$ is the following
\[\gpk{K_{3,4}} =
            \begin{cases}
                k+2 +2 \left\lfloor \frac{k}{8} \right\rfloor & k \equiv 0,1,2 \pmod{8} \\
                k+3 +2 \left\lfloor \frac{k}{8}\right\rfloor  & \text{otherwise}
            \end{cases}.\]
In Figure \ref{fig:K34ex}, we give $k$-rPDSs for $k=6,7,8$. The $k=6$ case is the first to improve the results from Theorem \ref{sjbound} and the change from $k=7$ to $k=8$ demonstrates when the Jump Condition is satisfied.           
\end{ex}

 \begin{figure}[htbp]
 \begin{center}
 \includegraphics[scale=1]{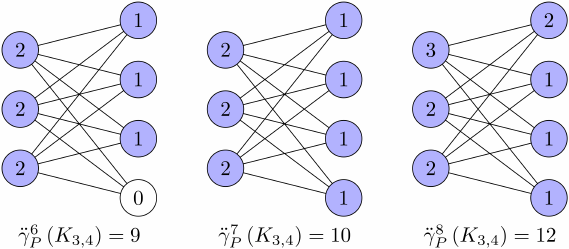}
 \end{center}
    \captionsetup{width=.9\linewidth}
 \caption{A demonstration of minimum $k$-rPDSs for $K_{3,4}$ when $k=6,7,8$}
 \label{fig:K34ex}
 \end{figure}

\section{Block Graphs}\label{sec:blocks}

For a connected graph $G$, a vertex $x \in V(G)$ is called a \emph{cut--vertex} if when we remove $x$ and its corresponding edges the resultant graph, denoted by $G-x$, is disconnected. A maximal connected induced subgraph of $G$ without a cut--vertex is called a \emph{block}.  The graph $G$ is called a \emph{block graph} if every block is complete. Block graphs are a superclass of trees. 

A \emph{rooted tree} is a tree in which one vertex designated as the \textit{root}. Suppose two vertices $u$ and $w$ are in a rooted tree with root $r$. If $u$ is on the $r-w$ path, we say that $w$ is a \textit{descendant} of $u$ and $u$ is an \textit{ancestor} of $w$. If $u$ and $v$ are also neighbors, we say that $v$ is a \emph{child} of $u$.

Let $G$ be a block graph with $t$ blocks $BK_1,BK_2 \dots BK_t$ and $p$ cut--vertices $x_1, x_2, \dots, x_p$. For each block $BK_i$ of $G$, define $B_i = \{w_j:w_j \in BK_i \text{ and $w_j$ is not a cut--vertex}\}$. In \cite{xksz06}, Xu et al. defined the \emph{refined cut--tree of $G$} to be the tree with vertex set $V^B=\{w_1^B,w_2^B,\dots,w_t^B,x_1,x_2,\dots,x_p\}$ where $w_i^B$ is a vertex representing the set $B_i$, called a \emph{block--vertex}, and the edge set is $E^B=\{(w_i^B,x_j):x_j\in BK_i\}$. 
Block--vertices will be referred to by their cardinality.
\begin{enumerate}
\item $w_i^B$ is called an \emph{empty--block} if $B_i=\emptyset$.
\item $w_i^B$ is called a \emph{one--block} if $|B_i|=1$.
\item $w_i^B$ is called a \emph{multi--block} if $|B_i| \geq 2$.
\end{enumerate}

In Figure \ref{fig:blockgraph}, we give an example of a block graph labelled in both $BK_i$ notation and $B_i$ notation with each of the blocks circled. Figure \ref{fig:refinedcuttree} shows the refined cut--tree of the graph in Figure \ref{fig:blockgraph}.

 \begin{figure}[hbtp]
 \begin{center}
 \includegraphics[scale=1]{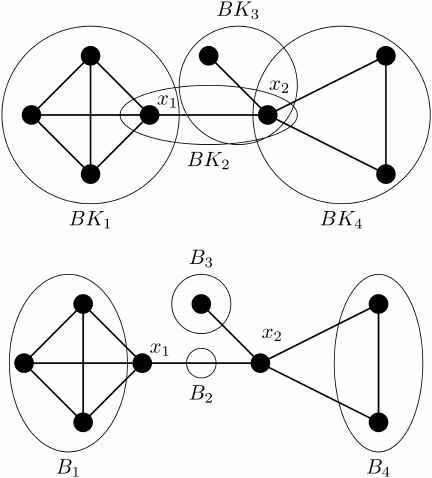}
 \end{center}
    \captionsetup{width=.9\linewidth}
 \caption{A block graph, labelled in both $BK_i$ notation and $B_i$ notation. Observe that $BK_2$ is made of only cut vertices and so the corresponding $B_2$ is an empty-block. Block $B_3$ is a one-block. Blocks $B_1$ and $B_3$ are multi-blocks}
 \label{fig:blockgraph}
 \end{figure}

 \begin{figure}[hbtp]
 \begin{center}
 \includegraphics[scale=1]{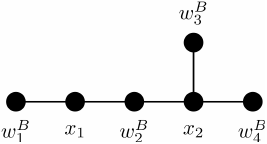}
 \end{center}
    \captionsetup{width=.9\linewidth}
 \caption{The refined cut-tree for the block graph in Figure \ref{fig:blockgraph}}
 \label{fig:refinedcuttree}
 \end{figure}

In \cite{xksz06}, Xu et. al. defined a liner--time algorithm on the refined cut--tree to find a minimum power dominating set of a block graph. Their algorithm processes the vertices in the refined cut--tree one at a time from the leaves to the root, but it does not necessarily finish the zero forcing steps once a branch has been processed. In this section power domination rules on the refined cut--tree are formally defined that are equivalent to power domination in the block graph using cut--vertices and clarify the arguments in \cite{xksz06}. 

In block graphs, a cut--vertex efficiently dominates the blocks in which it is contained, which leads to the following result.

\begin{lem}[\cite{xksz06}, Lemma 2.1]\label{cutvertex}
Let $G$ be a block graph. Then there exists a minimum power dominating set of $G$ in which every vertex is a cut--vertex of $G$.
\end{lem}

This can be generalized to PMU-defect-robust power domination as follows.

\begin{cor}\label{cutvertexrobust}
For a block graph $G$, there exists a minimum $k$-rPDS of $G$ in which every vertex is a cut--vertex of $G$.
\end{cor}

The refined cut--tree of $G$ captures the essential connections between the blocks of $G$, and the following definition explains how to translate power domination with cut--vertices in $G$ to the refined cut--tree.

\begin{defn}
Given a block graph $G$ and its refined cut--tree $T^B=(V^B,E^B)$, let $S \subseteq V^B$ be an initial set of cut--vertices. Define $CTS$ to be the set of observed vertices of $T^B$, and initially set $CTS=S$. Observe the vertices of $T^B$ according to the following observation rule:
\begin{enumerate}
\item Domination: For every $x_j \in S$, and for all $w_i^B \in N(x_j)$, $CTS=CTS \cup N[w_i^B]$.
\item Zero forcing through empty--blocks: If there exists an observed cut--vertex $x_j \in CTS$ such that an empty--block $w_i^B$ is the only unobserved neighbor of $x_j$, and $w_i^B$ has at most one unobserved neighbor, then $CTS = CTS \cup N[w_i^B]$.
\item Zero forcing through one--blocks: If there exists an observed cut--vertex $x_j$ such that a one--block $w_i^B$ is the only unobserved neighbor of $x_j$ and every  $v \in N(w_i^B)$ is observed, $CTS = CTS \cup \{w_i^B\}$.
\end{enumerate}
Repeat the zero forcing steps until no further changes are possible. If $CTS=V^B$ at the end of the process, the set $S$ is said to be a \emph{power dominating set for the cut--tree $T^B$}. 
\end{defn} 

Note that power domination on the refined--cut tree is not the same as power domination on a tree. For example, for the refined cut--tree in Figure \ref{fig:refinedcuttree}, a minimum power dominating set for the cut--tree is $\{x_1,x_2\}$ because the multi--block $w^B_1$ is only observed if its cut--vertex neighbor $x_1$ is in the power dominating set. However, a minimum power dominating set for the graph is $\{x_2\}$, which shows the difference between power domination and power domination for refined cut--trees.

\begin{lem}\label{PDblockcuttree}
Given a block graph $G=(V,E)$ and its corresponding refined cut--tree $T^B=(V^B,E^B)$, a set $S \subseteq V(G)$ in which every vertex is a cut--vertex of $G$ is a power dominating set of $G$ if and only if the corresponding set $S \subseteq V^B(T^B)$ is a power dominating set for the refined cut--tree $T^B$. 
\end{lem}

\begin{proof}
Let $x \in S$ and note that $x$ is a cut-vertex of $G$. Recall that $w^B_i \in N_{T^B}(x)$ implies $x \in BK_i$ and $N_{T^B}[w^B_i]=\{v \in V(G):v \in BK_i\}$. In the refined cut--tree, $x$ observes $N_{T^B}[w^B_i]$ for all $w^B_i \in N_{T^B}(x)$.  In other words, $x$ observes the vertices in $T^B$ corresponding to the neighbors of $x$ in $G$. Therefore, a vertex is observed by the domination step for $S$ in $G$ if and only if its corresponding vertex is observed in the domination step for $S$ on $T^B$. To show that the power domination process is equivalent on the block graph and the refined cut--tree, it is sufficient to show that if a vertex is observed by a zero forcing step in $G$ if and only if its corresponding vertex is observed by a zero forcing step in $T^B$. 

It is necessary to first show that if the initial set $S$ contains only cut--vertices in $G$ and $v \in V(G)$ is not a cut--vertex, then $v$ cannot perform a zero forcing step. If $v$ is observed and $v \in BK_i$, then $v$ had to be observed by one of its neighbors $y \in BK_i$. If $y \in S$, all of $BK_i$ is observed and $v$ has no unobserved neighbors. If $y$ is not in $S$ and $v$ has an unobserved neighbor $z$, then $z$ is also adjacent to $y$ and $y$ cannot force $v$ while $z$ is not observed. So, $v$ cannot be observed and also have a unique unobserved neighbor. Therefore $v$ cannot perform a zero forcing step. Moreover,if $v$ not a cut--vertex and is the only unobserved neighbor of a cut--vertex $x$, then the block--vertex $w_k^B \in V(T^B)$ containing $v$ is a one--block. 

Next we consider how zero forcing on $G$ corresponds to zero forcing on $T^B$. Assume that $x \in V(G)$ is an observed cut--vertex and has exactly one unobserved neighbor $v$. Without loss of generality, let $BK_0$ be the block of $G$ such that $x,v \in BK_0$. Let $BK_i$, $1 \leq i \leq d$ be the other $d$ blocks containing $x$. By assumption, $BK_i$ is observed for all $1 \leq i \leq d$ and  $BK_0-v$ is observed. By the argument above, this implies $v$ is either a cut--vertex and all vertices in $BK_0$ are cut--vertices or $v$ is the only vertex in $BK_0$ which is not a cut vertex. Translating to the refined cut--tree, $x$ is observed and all the block vertex neighbors of $x$ are observed except $B_0$. If $v$ is a cut--vertex, $B_0$ is an empty--block with exactly one unobserved neighbor, $v$, and $x$ can observe $B_0$ and $v$. If $v$ is not a cut--vertex, $B_0=\{v\}$ and $B_0$ has all blue neighbors, so $x$ can observe $B_0$. 

Conversely, if $x \in V^B$ is a cut--vertex able to observe an empty--block, then $x$ is observed and has one unobserved empty--block--vertex neighbor $B_0$ such that $B_0$ has at most one unobserved cut--vertex neighbor. In the block graph, this implies that the associated block $BK_0$ has one unobserved vertex neighbor of $x$. Since the other blocks in $T^B$ that $x$ is an element of are observed and the power domination rules observe a block in $T^B$ if and only if all of the cut--vertex neighbors of the block are observed, $x$ has no unobserved neighbors in any other block--vertex neighbors. Therefore, in $G$, $x$ has at most one unobserved neighbor, and $x$ can force the cut--vertex corresponding to the unobserved neighbor of $B_0$ in $T^B$. If $x \in V^B$ is a cut--vertex able to observe a one--block, then $x$ is observed and its only unobserved neighbor is the one--block $B_1$. Then all of the cut--vertex neighbors of $B_1$ are observed, indicating that the one vertex in block $BK_1$ that is not observed is the one vertex $v$ that is not a cut--vertex. Since the other block--vertex neighbors of $x$ in $T^B$ are observed, all of the other block--vertex neighbors of $x$ in $G$ are observed, and $x$ can observe its one unobserved neighbor $v$ in $G$. 
\end{proof}

The same argument can be used to translate PMU-defect-robust power domination between a block graph and its corresponding refined cut--tree.

\begin{cor} Given a block graph $G$, a multiset $M \subseteq V(G)$ in which every vertex is a cut--vertex is a minimum $k$-rPDS of $G$ if and only if $M \subseteq V^B(T^B)$ is a minimum $k$-rPDS of the corresponding refined cut--tree $T^B(G)$. 
\end{cor}

The refined cut--tree power domination rules allows a generalization of the procedure for finding a minimum power dominating set for trees in \cite{hhhh02} to find $\gpk{G}$ for a block graph $G$.

\begin{thm}
Given a block graph $G$, $\gpk{G}=(k+1)\gp{G}$.
\end{thm}

\begin{proof}
Let $G$ be a block graph. By Proposition \ref{basebounds}, $\gpk{G} \leq (k+1)\gp{G}$. By way of contradiction, assume that $\gpk{G}<(k+1)\gp{G}$. By Corollary \ref{cutvertexrobust}, there exists a minimum $k$-rPDS of $G$ in which every vertex is a cut--vertex. Choose such a minimum $k$-rPDS, $S$, such that the number of vertices $x \in S$ with $\pmus{x} \leq k$ is minimum. Define $A=\{x \in S:\pmus{x} \leq k\}$. Choose a vertex $r \in A$ and create the refined cut--tree $T^B(G)$ rooted at $r$. 

Let $z$ be a vertex in $A$ at the maximum distance from the root of $T^B$, that is, $d(z, r)=\max\{d(x,r):x \in A\}$. Since $S$ is a $k$-rPDS multiset of $T^B$, all of the PMUs on $z$ may be removed and the resulting set is still a power dominating set. This implies that observing the descendants of $z$ needs at most one block--vertex child of $z$ to be observed by zero forcing from $z$. Create the multiset $S'$ in which all of the PMUs on $z$ are moved to its nearest cut--vertex ancestor $y$, $S'=(S \setminus \{z^{\pmus{z}}\}) \cup \{y^{\pmus{z}}\}$. Let $w_i^B$   be the block--vertex adjacent to both $z$ and $y$. By the domination rule of power domination on the refined cut--tree, a PMU on $y$ observes $N[w_i^B]$ and thus $z$ is observed and has at most one unobserved neighbor so $z$ is able to perform a zero forcing step. Therefore, $S'$ is a power dominating set of the refined cut--tree. If $\pmus{y}_{S'}>k+1$, the set $S'$ is not a minimum $k$-rPDS, contradicting the fact that $|S|=|S'|=\gpk{G}$. If $\pmus{y}=k+1$ in $S'$, then $S'$ is a minimum $k$-rPDS of $G$ with fewer vertices with $\pmus{x} \leq k$, contradicting the choice of $S$. If $\pmus{y}<k+1$, then the PMUs on $y$ are not needed to dominate its descendants and the PMUs may be moved up to the nearest cut--vertex ancestor of $y$ while still retaining the ability to power dominate the refined cut--tree. Continue moving the PMUs up to the nearest cut--vertex ancestor. If no contradiction is reached, all of the PMUs are on the root $r$ and $\pmus{r} \leq k$. However, this would imply that $\gpk{G} \leq k$, a contradiction.
\end{proof}

Every tree is a block graph, so the result for trees is immediate.

\begin{cor}
Given a tree $T$, $\gpk{T}=(k+1)\gp{T}$.
\end{cor}

\section{Future Work}\label{sec:futurework}

We have defined the $k$-PMU-defect-robust power domination number, determining an upper bound and exact results for some families of graphs. There are still interesting questions to explore.

First, as we see in Example~\ref{ex:gridgraphs}, the bound from Theorem~\ref{sjbound} is an improved upper bound for grid graphs, however, this bound is not tight for all $k$. What is $\gpk{G_{6,6}}$, and moreover, what is $\gpk{G_{n,n}}$ for other values of $n$? Moreover, are there any examples in which the bound from Theorem \ref{sjbound} is tight for infinitely many $k$, or can one show that Observation \ref{obs:sjboundnottight} implies that there is some sufficiently large $k^*$ for which the bound is never tight for $k > k^*$?

In Figure \ref{fig:sequalsnex}, we demonstrate graphs for which $\gp{G}=2$ and $\bigpds{G}=|V(G)|$. The first graph shown in Figure \ref{fig:sequalsnex} is a circulant graph---could this be shown for circulant graphs in general? Can we find any examples of this behavior for graphs with $\gp{G} > 2$? Lemma \ref{bigpdsNimplies} gives one characteristic of such a graph, but can we determine more such properties?

\section*{Acknowledgements}

The authors thank Johnathan Koch of Applied Research Solutions for his help in finding the graphs for Figure \ref{fig:sequalsnex}. This project was sponsored, in part, by the Air Force Research Laboratory via the Autonomy Technology Research Center and Wright State University. This research was also supported by Air Force Office of Scientific Research award 23RYCOR004.

\bibliographystyle{plain}
\bibliography{robust.bib}

\end{document}